\documentclass[11pt]{article}

\usepackage{amsmath, amscd, amssymb,amsfonts,amsthm,amscd,graphicx,psfrag,epsfig}
\usepackage{array, float, mathrsfs, mathtools, stmaryrd, youngtab}
\usepackage[all]{xy}

 \usepackage{geometry}               
\usepackage[
colorlinks=true,%
linkcolor=blue,%
citecolor=blue,%
filecolor=blue,%
menucolor=blue,%
urlcolor=blue]{hyperref}

\newtheorem{theorem}{Theorem}[section]

\newtheorem{lemma}[theorem]{Lemma}
\newtheorem{proposition}[theorem]{Proposition}
\newtheorem{corollary}[theorem]{Corollary}

\newtheorem{conjecture}[theorem]{Conjecture}
\newtheorem{theorem-definition}[theorem]{Theorem-Definition}
\newtheorem{theorem-construction}[theorem]{Theorem-Construction}
\newtheorem{lemma-definition}[theorem]{Lemma--Definition}
\newtheorem{lemma-construction}[theorem]{Lemma--construction}
\newtheorem{definition}[theorem]{Definition}

\theoremstyle{definition}

\newtheorem{remark}[theorem]{Remark}
\newtheorem{example}[theorem]{Example}

\newtheorem*{acknowledgement}{Acknowledgements}

\newcommand{\bg}{\begin{equation}\begin{gathered}}
\newcommand{\eg}{\end{gathered}\end{equation}}

\newcommand{\n}{\mathfrak{n}}

\newcommand{\C}{\mathbb{C}}

\newcommand{\Z}{\mathbb{Z}}

\newcommand{\hlra}{\lhook\joinrel\longrightarrow}

\newcommand{\old}[1]{}

\newcommand{\B}{{\rm B}}
\newcommand{\G}{{\rm G}}
\newcommand{\U}{{\rm U}}
\renewcommand{\H}{{\rm H}}

\newcommand{\bS}{{\mathbb{S}}}
\newcommand{\lms}{\longmapsto}
\newcommand{\lra}{\longrightarrow}

\newcommand{\be}{\begin{equation}}
\newcommand{\ee}{\end{equation}}
\newcommand{\bt}{\begin{theorem}}
\newcommand{\et}{\end{theorem}}
\newcommand{\bd}{\begin{definition}}
\newcommand{\ed}{\end{definition}}
\newcommand{\bp}{\begin{proposition}}
\newcommand{\ep}{\end{proposition}}

\newcommand{\bl}{\begin{lemma}}
\newcommand{\el}{\end{lemma}}
\newcommand{\bc}{\begin{corollary}}
\newcommand{\ec}{\end{corollary}}
\newcommand{\bcon}{\begin{conjecture}}
\newcommand{\econ}{\end{conjecture}}
\newcommand{\la}{\label}

\usepackage{tikz}
\usepackage{tikz-cd}
\usetikzlibrary{arrows,chains,shapes,matrix,positioning,scopes} 
\usetikzlibrary{decorations.markings}
\tikzstyle arrowstyle=[scale=1.1]
\tikzstyle directed=[postaction={decorate,decoration={markings,
    mark=at position 1 with {\arrow[arrowstyle]{latex}}}}]
\tikzstyle reverse directed=[postaction={decorate,decoration={markings,
    mark=at position .45 with {\arrowreversed[arrowstyle]{latex};}}}]

\title{Duals of semisimple Poisson-Lie groups and cluster theory of moduli spaces of {$\G$}-local systems}
\author{Linhui Shen}
\date{}

\begin{document}

\maketitle

\begin{abstract}
We study the dual $\G^\ast$ of a standard semisimple  Poisson-Lie group $\G$ from a perspective of cluster theory. We show that the coordinate ring $\mathcal{O}(\G^\ast)$ can be naturally embedded into  a quotient algebra of a cluster Poisson algebra with a Weyl group action. The coordinate ring $\mathcal{O}(\G^\ast)$ admits a natural basis, which has positive integer structure coefficients and satisfies an invariance property under a braid group action. We continue the study of the moduli space $\mathscr{P}_{\G,\bS}$ of $\G$-local systems introduced in \cite{GS3} and prove that the coordinate ring of $\mathscr{P}_{\G, \bS}$ coincides with its underlying cluster Poisson algebra. 
\end{abstract}

\tableofcontents

\section{Introduction}
Let $\G$ be a connected and adjoint complex semisimple Lie group equipped with the standard Poisson-Lie structure. Denote by $\G^\ast$ the  {\it dual Poisson-Lie} group of $\G$. Our main goal is to study the Poisson-Lie structure of $\G^\ast$ in the context of cluster theory. 

The coordinate ring $\mathcal{O}(\G^\ast)$ of $\G^\ast$ is a Poisson algebra, whose quantization  gives rise to the quantum group $\mathcal{U}_q(\mathfrak{g})$. In general, the Poisson algebra $\mathcal{O}(\G^\ast)$ may not carry a compatible cluster structure. To resolve this issue, one approach is to extend the family of cluster algebras. For example, Gekhtman, Shapiro, and Vainstein showed in \cite{GSVgc} that the dual ${\rm GL}_n^\ast$ carries a compatible  {\it generalized cluster structure}. In this paper, we propose a different approach by embedding the Poisson algebra $\mathcal{O}(\G^\ast)$ into a quotient algebra of a cluster Poisson algebra together with a Weyl group action.  

Our main tool is the  moduli spaces of $\G$-local systems studied in \cite{GS3}.  Let $\bS$ be an oriented surface with punctures and marked boundary points. The moduli space $\mathscr{P}_{\G, \bS}$ parametrizes $\G$-local systems over $\bS$ together with some boundary data (Definition \ref{sdnion}). Theorem 1.12 of \cite{GS3} asserts that the space $\mathscr{P}_{\G, \bS}$ carries a natural cluster Poisson structure, which gives rise to a cluster Poisson algebra $\mathcal{O}_{\rm cl}(\mathscr{P}_{\G, \bS})$ as defined in \eqref{sdafionio}. 

We prove the following stronger result in Section \ref{sec.3.3}.
\bt 
\label{cdsn.cluster.poisson}
The cluster Poisson algebra $\mathcal{O}_{\rm cl}(\mathscr{P}_{\G, \bS})$ coincides with the ring $\mathcal{O}(\mathscr{P}_{\G, \bS})$ of regular functions on $\mathscr{P}_{\G, \bS}$.
\et

Of particular interest to us is the case when $\bS=\odot$ is a once-punctured disk with two marked points on its boundary. Our starting point is Theorem 2.25 in \cite{GS3}, which asserts that the dual group $\G^\ast$ is identified with the moduli space $\mathscr{L}_{\G, \odot}$. 
As established in \cite{GS3}, the Weyl group $W$ acts on $\mathscr{P}_{\G, \odot}$ as {\it cluster Poisson transformations}. The subalgebra $\mathcal{O}(\mathscr{P}_{\G, \odot})^W$ of $W$-invariant regular functions on $\mathscr{P}_{\G, \odot}$ inherits a Poisson structure from $\mathcal{O}(\mathscr{P}_{\G, \odot})$. 
 The \emph{outer monodromies} $\mu_1, \ldots, \mu_r$ are  Poisson central elements in $\mathcal{O}(\mathscr{P}_{\G, \odot})^W$ that correspond to the simple positive roots of $\G$. Let $\mathcal{I}$ be the ideal generated by $\mu_i-1$ for $i=1,\ldots, r$. 
One of our main results is as follows.

\bt \label{main.result.1}
The following algebras are isomorphic as Poisson algebras:
\be
\label{isomorphism}
\mathcal{O}(\G^\ast)\stackrel{\sim}{=} \mathcal{O}(\mathscr{P}_{\G, \odot})^W {\Big \slash} \mathcal{I}. 
\ee
 \et
Note that the isomorphism \eqref{isomorphism} does not automatically endow the algebra $\mathcal{O}(\G^\ast)$ with a cluster algebra structure. 
We emphasize that  \eqref{isomorphism} preserves  the Poisson structures on both algebras. The isomorphism \eqref{isomorphism} is the semiclassical limit of a canonical embedding
\be
\label{asncijosn}
\mathcal{U}_q(\mathfrak{g}) \longrightarrow \mathcal{O}_q(\mathscr{P}_{\G, \odot})^W {\Big \slash} \mathcal{I},
\ee
where $\mathcal{O}_q(\mathscr{P}_{\G, \odot})$ is a quantization of the cluster Poisson algebra $\mathcal{O}(\mathscr{P}_{\G, \odot})$ following Fock and Goncharov \cite{FGrep}. The embedding \eqref{asncijosn} is established in \cite[Theorem 2.22]{GS3}, which generalizes the cluster realization of $\mathcal{U}_q(\mathfrak{sl}_n)$ by Schrader and Shapiro \cite{SS}. Another closely related structure is the Poisson geometry of the Grothendieck-Springer resolution $\widetilde{\G}$ studied in  \cite{EL}. Inspired by the results of \cite{EL},  Brahami \cite{B} and Schrader-Shapiro \cite{SS1} study the cluster Poisson structures related to $\widetilde{\G}$ and $\G^*$. Both \cite{B} and \cite{SS1} crucially use the cluster Poisson structure on the biggest double Bruhat cell inside ${\rm G}$. Their connections to our Theorem \ref{main.result.1} is an interesting direction for future research. 

Theorem \ref{main.result.1} allows us to apply cluster theory to the study of $\mathcal{O}(\G^\ast)$.  Lusztig constructed an action of the  braid group $\mathbb{B}_\G$ on the quantum group $\mathcal{U}_q(\mathfrak{g})$, e.g., see \cite[Part VI]{Lus}. After the specialization $q=1$, the Braid group acts on $\G^\ast$ as Poisson automorphisms.   In Section \ref{braid.group.action.4}, we present an explicit description of this action when $\G^*$ is realized as a subgroup of $\B_+\times \B_-$. We apply Gross, Hacking, Keel, and Kontsevich's construction of theta bases of cluster algebras \cite{GHKK} in this situation and obtain a natural basis for  $\mathcal{O}(\G^\ast)$. 

\bt 
\la{main.result.2}
The algebra $\mathcal{O}(\G^\ast)$ admits a natural linear basis $\overline{\Theta}$ with positive integer structure coefficients. The basis $\overline{\Theta}$, as a set, is invariant under the braid group action. 
\et

We expect that Theorems \ref{main.result.1} $\&$  \ref{main.result.2}  can be generalized to the quantum setting. In particular, we conjecture that the quantum group $\mathcal{U}_q(\mathfrak{g})$ admits a natural basis invariant under the braid group action, which can be viewed as an extension of Lusztig's dual canonical basis for $\mathcal{U}_q(\mathfrak{n})$. We verify this conjecture for $\mathfrak{g}=\mathfrak{sl}_2$ via a straightforward calculation (Example \ref{acsdnjoqde}).

\begin{acknowledgement}
I wish to thank Alexander Goncharov for useful discussions that prompted me to finish the paper. I am grateful to the referees for very careful reading of this paper and for many useful suggestions. I was supported by the Collaboration Grants for Mathematicians from the Simons Foundation, $\#$ 711926.
\end{acknowledgement}

\section{Preliminaries}
\subsection{Cluster Poisson Algebras}
A cluster ensemble, introduced by Fock and Goncharov in \cite{FGensem},  is a pair $(\mathscr{X}, \mathscr{A})$ of positive spaces associated with the same class of combinatorial data called seeds. The coordinate ring of $\mathscr{A}$ coincides with the upper cluster algebra of Berenstein, Fomin, and Zelevinsky \cite{BFZ}. We  focus on the coordinate ring of $\mathscr{X}$, which carries a natural Poisson structure and is called the cluster Poisson algebra in this paper. It is worth mentioning that the cluster Poisson algebras are not the same as the upper cluster algebras. In this section, we rapidly recall the definition of cluster Poisson algebras and their basic properties. 

Denote by $\mathcal{F}$ the field of rational functions in  independent variables $x_1, \ldots, x_n$ with complex coefficients. Let $\widehat{\varepsilon}=(\widehat{\varepsilon}_{ij})$ be an $n\times n$ rational skew-symmetric matrix.
It gives rise to a Poisson bracket on $\mathcal{F}$ such that 
\be
\label{vfno}
\{x_i, x_j\}= 2 \widehat{\varepsilon}_{ij}x_ix_j.
\ee
Fix a positive integer $m\leq n$. The {\it multipliers}  $d_1, \ldots, d_m$  are positive integers such that $\varepsilon_{ik}:=\widehat{\varepsilon}_{ik}d_k\in \mathbb{Z}$ for $1\leq i\leq n$ and $1\leq k \leq m$. The  matrix $\varepsilon=(\varepsilon_{ik})$ is called an {\it exchange matrix}. The set ${\bf x}=\{x_1, \ldots, x_n\}$ is called a {\it cluster Poisson chart}. 
The pair $\Sigma=({\bf x}, \varepsilon)$  is called a {\it seed} in $\mathcal{F}$.

For $1\leq k\leq m$, we define the adjacent cluster Poisson chart ${\bf x}_k=\{x_1', \ldots, x_n'\}$ by
\[
x_i'=\left\{\begin{array}{ll}
    x_{k}^{-1} & \mbox{if } i=k; \\ x_i\left(1+x_k^{-{\rm sgn}(\varepsilon_{ik})}\right)^{-\varepsilon_{ik}}
     & \mbox{if } i\neq k .
\end{array} \right.
\]
The chart ${\bf x}_k$ gives a transcendence basis of  $\mathcal{F}$. The Poisson bracket \eqref{vfno} in terms of ${\bf x}_k$ becomes
$\{x_i', x_j'\}= 2 \widehat{\varepsilon}_{ij}' x_i'x_j'$. It determines an $n\times m$ integer matrix $\varepsilon'$ whose entries
are
\[
\varepsilon_{ij}':=\widehat{\varepsilon}_{ij}'d_j= \left\{\begin{array}{ll}
-\varepsilon_{ij} & \mbox{~~if $i=k$ or $j=k$};\\
\varepsilon_{ij}+\frac{|\varepsilon_{ik}|\varepsilon_{kj}+ \varepsilon_{ik}|\varepsilon_{kj}|}{2}& \mbox{~~otherwise}.
\end{array}\right.
\]
The process of obtaining the new seed $\Sigma_k=({\bf x}_k, \varepsilon')$ from $\Sigma$ is called a {\it seed mutation} in the direction $k$. We say that a seed $\Sigma'$ is mutation equivalent to $\Sigma$ and denote by $\Sigma'\sim \Sigma$ if it can be obtained from $\Sigma$ by a sequence of seed mutations.

\bd
\label{svafonv}
Let $\mathbb{C}[{\bf x}^\pm]=\mathbb{C}[x_1^{\pm 1}, \ldots, x_n^{\pm 1}]$ denote the ring of Laurent polynomials in $x_1, \ldots, x_n$. The {\it upper bound} associated with a seed $\Sigma$ is a ring
\[
\mathcal{U}(\Sigma)=\mathbb{C}[{\bf x}^{\pm}] \cap \mathbb{C}[{\bf x}_1^{\pm}]\cap \ldots \cap \mathbb{C}[{\bf x}_m^{\pm}].
\]
The {\it cluster Poisson algebra} associated with $\Sigma$ is the intersection of  Laurent polynomial rings for all seeds $\Sigma'$ that are mutation equivalent to $\Sigma$:
  \be
   \label{adbcniadj}
  \mathcal{O}(\Sigma)=\bigcap_{\Sigma'\sim \Sigma} \mathbb{C}[{\bf x}_{\Sigma'}^{\pm}]. 
  \ee
\ed

The following crucial lemma is a direct consequence of \cite[Theorem 3.9]{GHK}.
\begin{lemma}
\label{GHK.up}
 $\mathcal{O}(\Sigma)=\mathcal{U}(\Sigma)$.
\end{lemma}

Note that  $\mathcal{O}(\Sigma)$ is a Poisson algebra. A {\it quasi-cluster transformation} of $\mathcal{O}(\Sigma)$ is a Poisson automorphism of $\mathcal{O}(\Sigma)$ that can be obtained by a sequence of mutations followed by a seed isomorphism, e.g., see \cite[Def.A.29]{SWflag}.
Many cluster Poisson algebras contain natural linear bases that are invariant under quasi-cluster transformations. For example, borrowing ideas from mirror symmetry, Gross, Hacking, Keel, and Kontsevich \cite{GHKK} reformulated the seed mutations of cluster Poisson algebras in terms of wall-crossing structures (a.k.a.  scattering diagrams), and constructed a theta basis $\Theta$ by counting broken lines in the scattering diagram. Elements of $\Theta$ in general are formal power series. One sufficient condition to guarantee that $\Theta$ is a linear basis of $\mathcal{O}(\Sigma)$ is the existence of {\it maximal green} (or, more generally, {\it reddening}) sequences of seed mutations (\cite[Prop.0.7]{GHKK}).

\subsection{Poisson-Lie groups}
For the convenience of the reader, we recall several basic concepts about Poisson-Lie groups. See \cite[Ch.1]{ChPr} for a detailed exposition.

Let $\G$ be a Lie group. Its Lie algebra $\mathfrak{g}=Lie(\G)$ carries a Lie bracket $[,]: \mathfrak{g}\otimes\mathfrak{g}\rightarrow \mathfrak{g}$. The group $\G$ is said to be a Poisson-Lie group if it is also a Poisson manifold such that the multiplication map $\G \times \G \rightarrow \G$ is a Poisson map. The Poisson structure on $\G$ induces a Lie bracket $\delta$ on the dual space $\mathfrak{g}^\ast$ that is compatible with the bracket $[,]$. The triple $(\mathfrak{g}, [,], \delta)$ is called  a Lie bialgebra. The main theorem of Poisson-Lie theory, due to Drinfeld, asserts that the functor $F: \G \rightarrow Lie(\G)$ is an equivalence between the category of connected and simply connected Poisson-Lie groups and the category of Lie bialgebras. 
For a finite dimensional Lie bialgebra $(\mathfrak{g}, [,], \delta)$, its dual $(\mathfrak{g}^\ast, \delta^\ast, [,]^\ast)$ is  a Lie bialgebra. In the group setting, we say the Poisson-Lie groups $\G$ and $\G^\ast$ are dual to each other if their corresponding Lie bialgebras are dual to each other. 

Now assume $\G$ is a connected complex semisimple Lie group. Its Lie algebra $\mathfrak{g}$ carries a Cartan-Killing form $B(\ast,\ast)$. Let $\mathfrak{b}_+$ and $\mathfrak{b}_-$ be a pair of opposite Borel subalgebras of $\mathfrak{g}$ so that their intersection $\mathfrak{h}=\mathfrak{b}_+\cap \mathfrak{b}_-$ is a Cartan subalgebra. 
The {\it Drinfeld double} $D(\mathfrak{g}):=\mathfrak{g}\oplus \mathfrak{g}$ admits an invariant non-degenerate bilinear form 
\be
\label{wbasnjonvj}
\left\langle (x, y), ~(x',y')\right\rangle := B(x, x')- B(y,y').
\ee
Let $\mathfrak{p}_+$ be the diagonal Lie subalgebra of $D(\mathfrak{g})$. There is another Lie subalgebra
\be
\label{cwnionwv}
\mathfrak{p}_-=\left\{(x,y)\in \mathfrak{b}_+\times \mathfrak{b}_-~\middle |~ \mbox{the $\mathfrak{h}$-component of $x+y$ is 0} \right\}.
\ee
Note that $D(\mathfrak{g})=\mathfrak{p}_+ \oplus \mathfrak{p}_-$ as vector spaces. Both $\mathfrak{p}_+$ and $\mathfrak{p}_-$ are isotropic with respect to the bilinear form
\eqref{wbasnjonvj}. The datum $(D(\mathfrak{g}), \mathfrak{p}_+, \mathfrak{p}_-)$ is called a {\it Manin triple}, which determine a Lie bialgebra structure on $\mathfrak{g}\stackrel{\sim}{=}\mathfrak{p}_+$ in the following way.
The bilinear form \eqref{wbasnjonvj} induces an isomorphism $\mathfrak{p}_-\stackrel{\sim}{=}\mathfrak{p}^\ast_+$. The Lie algebra structure on $\mathfrak{p}_-$ corresponds to a compatible Lie coalgebra structure on the dual $\mathfrak{p}_+$. Putting them together, we get a Lie bialgebra structure on $\mathfrak{p}_+=\mathfrak{g}$, which further determines the {\it standard Poisson-Lie structure} on $\G$.

The Manin triple determines a Lie bialgebra structure on $\mathfrak{g}^\ast\stackrel{\sim}{=}\mathfrak{p}_-$ in a similar way.
Let $\B_+$ and $\B_-$ be the Borel subgroups of $\G$ corresponding to $\mathfrak{b}_+$ and $\mathfrak{b}_-$ respectively. The dual Poisson-Lie group $\G^\ast$ of $\G$ is a subgroup of $\B_+\times \B_-$ such that  $Lie(\G^\ast)=\mathfrak{p}_-$.

\subsection{Quantization of Poisson algebras}
Fix a positive integer $d$.
Let $\mathbb{K}=\mathbb{C}[q^{\pm 1/d}]$ be the ring of Laurent polynomials in the variable $q^{1/d}$. Let $\mathcal{A}_q$ be a possibly non-commutative $\mathbb{K}$-algebra. We assume that $\mathcal{A}_q$ is free as a $\mathbb{K}$-module and that 
\be
\label{sdcmniw0efn}
\forall f, g \in \mathcal{A}_q, \hskip 10mm [f,g]:=fg -gf \in (q^{1/d}-1)\mathcal{A}_q.
\ee
Set $\mathcal{A}= \mathcal{A}_q/(q^{1/d}-1)\mathcal{A}_q$. Denote by $f\mapsto \overline{f}$ the projection from $\mathcal{A}_q$ to $\mathcal{A}$. By \eqref{sdcmniw0efn}, $\mathcal{A}$ becomes a commutative algebra and carries a  bilinear map $\{,\}:~\mathcal{A}\times \mathcal{A}\rightarrow \mathcal{A}$ such that
\[
\left\{ \overline{f}, \overline{g}\right\} = \overline{\frac{1}{d(q^{1/d}-1)}[f,g]}, ~~~~~~\forall f, g\in \mathcal{A}_q.
\]
It is easy to check that the bracket $\{,\}$ is skew-symmetric and satisfies the Jacobi identity and the Leibniz identity. Therefore $\mathcal{A}$ is a commutative Poisson algebra. 
The algebra $\mathcal{A}$ is called the {\it semiclassical limit} of $\mathcal{A}_q$. Conversely, $\mathcal{A}_q$ is called a {\it quantization} of $\mathcal{A}$.

Recall the dual Poisson-Lie group $\G^\ast$. Its coordinate ring $\mathcal{O}(\G^\ast)$ is a Poisson algebra. Let $\mathcal{U}_q(\mathfrak{g})$ be the quantum group associated to $\mathfrak{g}$. The quantum group $\mathcal{U}_q(\mathfrak{g})$ can be realized as a free $\mathbb{K}$-algebra over a suitable rescaled PBW basis. See \cite[Sec.11.4]{GS3} for a geometric interpretation of the rescaled PBW basis.  Following \cite[Th. 7.6]{CKP}, the semiclassical limit of $\mathcal{U}_q(\mathfrak{g})$ is $\mathcal{O}(\G^\ast)$. 

\begin{example}
\label{scvdnfqpvnwqn}
Let $\mathbb{K}=\mathbb{C}[q^{\pm 1/2}]$. The quantum group $\mathcal{U}_{q}(\mathfrak{sl}_2)$ is a $\mathbb{K}$-algebra generated by ${\bf E}, {\bf F}, {\bf K}^{\pm1}$ and satisfies the relations
\[
{\bf K}{\bf E} =q^2 {\bf E}{\bf K}, \hskip 7mm {\bf K}{\bf F} =q^{-2} {\bf F}{\bf K}, \hskip 7mm {\bf E}{\bf F}-{\bf F}{\bf E}=(q-q^{-1})({\bf K}^{-1}-{\bf K}).
\]
It semiclassical limit becomes $\mathcal{O}({\rm PGL}_2^\ast)\stackrel{\sim}{=}\mathbb{C}[e, f, k, k^{-1}]$ with the Poisson bracket
\[
\{k, e\}= 2ek, \hskip 7mm \{k, f\} =-2fk, \hskip 7mm \{e,f\}=2(k^{-1}-k).
\]
\end{example}

Recall the cluster Poisson algebra $\mathcal{O}(\Sigma)$ in Definition \ref{svafonv}. Let $d$ be the double of the least common multiple of the multipliers $d_1, \ldots, d_m$ and let $\mathbb{K}=\C[q^{\pm 1/d}]$. To every seed $\Sigma'\sim \Sigma$ is associated a quantum torus algebra $\mathcal{T}(\Sigma')$ over $\mathbb{K}$, with generators ${{\rm X}_1'}^{\pm1}, \ldots, {{\rm X}_n'}^{\pm1}$, and the relations
\[
{\rm X}_i'{\rm X}_j'=q^{2\widehat{\varepsilon}_{ij}'}{\rm X}_i'{\rm X}_j'.
\]
The algebra $\mathcal{T}(\Sigma')$ is a quantization of the Laurent polynomial ring $\mathbb{C}[{\bf x}_{\Sigma'}^{\pm}]$. 
By utilizing the quantum dilogarithm, Fock and Goncharov introduced a quantization of the cluster Poisson transformations in \cite{FGrep}. Therefore we get a $\mathbb{K}$-algebra
\be
\label{evfoqscsvnj}
\mathcal{O}_q(\Sigma):=\bigcap_{\Sigma'\sim \Sigma}\mathcal{T}(\Sigma').
\ee

Note that $\mathbb{K}$ is a principal ideal domain and $\mathcal{T}(\Sigma')$ is a free $\mathbb{K}$-module. Therefore $\mathcal{O}_q(\Sigma)$ is free as a submodule of $\mathcal{T}(\Sigma')$. Consider the semiclassical limit
\[\mathcal{O}_1(\Sigma):=\mathcal{O}_q(\Sigma){\Big \slash} (q^{1/d}-1)\mathcal{O}_q(\Sigma).
\]
There is a natural injective Poisson algebra homomorphism  
\be
\label{sacfdouncdfoq}
\mathcal{O}_1(\Sigma)\hlra\mathcal{O}(\Sigma).
\ee
\begin{remark}
 We conjecture that \eqref{sacfdouncdfoq} is an isomorphism and therefore  $\mathcal{O}_q(\Sigma)$ is a quantization of the cluster Poisson algebra $\mathcal{O}(\Sigma)$. Note that both sides of \eqref{sacfdouncdfoq} consider the rings of ``universally Laurent polynomials" (see \eqref{adbcniadj} and \eqref{evfoqscsvnj}). The injectivity of \eqref{sacfdouncdfoq} follows as an easy consequence. But the surjectivity part seems to be very hard. One possible approach is first to establish a quantum analog of Lemma \ref{GHK.up}.

There is a lot of literature studying the specializations of quantum cluster algebras. For example, see \cite{GLS}. It is worth mentioning that the paper \cite{GLS} investigates the quantum cluster algebras  generated by quantum cluster variables. The surjectivity is easy in this setting. On the contrary, the injectivity becomes the main challenge. 

\end{remark}

\section{Moduli spaces of $\G$-local systems}
In this section we rapidly review the definition and elementary properties of the moduli space $\mathscr{P}_{\G, \bS}$ following \cite{GS3}. 

\subsection{Definition}
The flag variety $\mathcal{B}$ consists of Borel subgroups of $\G$. We say that a pair $\B, \B'\in \mathcal{B}$ is generic if their intersection $\H:=\B\cap \B'$ is an abelian subgroup. Let $I=\{1,\ldots, r \}$ parametrize the simple positive coroots $\alpha_1^\vee, \ldots, \alpha_r^\vee$ of $\G$. The datum $p=(\B, \B', x_i, y_i; i\in I)$ is called a \emph{pinning} over a generic pair $(\B, \B')$ if it assigns to every $i\in I$ a homomorphism $\gamma_i: {\rm SL}_2 \rightarrow \G$ such that \[
\gamma_i \begin{pmatrix} 1 & a\\
0 & 1\\\end{pmatrix} = x_i(a)\in \B, \hskip 5mm \gamma_i \begin{pmatrix} 1 & 0\\
a & 1\\\end{pmatrix} = y_i(a)\in \B', \hskip 5mm \gamma_i \begin{pmatrix} a & 0\\
0 & a^{-1}\\\end{pmatrix} = \alpha^\vee_i(a)\in \H.
\]
The space of pinnings is a left principal homogeneous $\G$-space, also known as a left $\G$-torsor. We sometimes will fix a pinning $p$ and refer it as a {\it standard pinning}. 

Let $\bS$ be an oriented topological surface with interior punctures and at least one marked point on each of its boundary circles. The marked points cut the boundary circles into disconnected intervals which are called  boundary intervals.

A $\G$-local system $\mathcal{L}$ over $\bS$ is a principal $\G$-bundle over $\bS$ with flat connections. Denote by $\mathcal{L}_{\mathcal{B}}:=\mathcal{L}\times_{\G}\mathcal{B}$ its associated $\mathcal{B}$-bundle.

\bd 
\label{sdnion}
The  space $\mathscr{P}_{\G, \bS}$ parametrizes $\G$-orbits of the data $(\mathcal{L}, \{\B_c\}, \{\B_i\}, \{p_e\})$, where
\begin{itemize}
    \item $\mathcal{L}$ is a $\G$-local system over $\bS$;
    \item for every puncture $c$, the data $\B_c$ is a flat section of $\mathcal{L}_{\mathcal{B}}$ over the circle around $c$;
    \item for every marked point $i$, the data $\B_i$ is a section of $\mathcal{L}_{\mathcal{B}}$ over $i$;
    \item for every boundary interval $e$ with endpoints $i$ and $j$, the associated pair $(\B_i, \B_j)$ is generic, and $p_e$ is a pinning over $(\B_i, \B_j)$.
\end{itemize}
\ed

\begin{example}
Let $\bS=D_n$ be a regular $n$-gon with marked boundary points sitting on its vertices. Since $\pi_1(D_n)=0$,  the local systems over $D_n$ are trivial. The space $\mathscr{P}_{\G, D_n}$ parametrizes $\G$-orbits of the tuples $(\B_1,...,\B_n; p_1,..., p_n)$ such that every $p_i$ ($i \in \Z/n\Z$) is a pinning over the generic pair $(\B_i,\B_{i+1})$.
Elements of $\mathscr{P}_{\G, D_5}$  are pictured as follows.
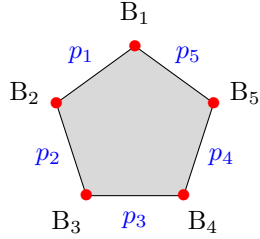
\begin{center}
\begin{tikzpicture}[scale=1.1]
\draw[fill=gray!30] (90:1) -- (162:1) -- (234:1) -- (306:1) -- (18:1) -- (90:1);
\node[red]  at (90:1) {$\bullet$};
\node[red]  at (162:1) {$\bullet$};
\node[red]  at (234:1) {$\bullet$};
\node[red]  at (306:1) {$\bullet$};
\node[red]  at (18:1) {$\bullet$};
\node  at (90:1.4) {\small $\B_1$};
\node  at (162:1.4) {\small $\B_2$};
\node  at (234:1.4) {\small $\B_3$};
\node  at (306:1.4) {\small $\B_4$};
\node  at (18:1.4) {\small $\B_5$};
\node[blue]  at (126:1.1) {\small $p_1$};
\node[blue]  at (198:1.1) {\small $p_2$};
\node[blue]  at (270:1.1) {\small $p_3$};
\node[blue]  at (342:1.1) {\small $p_4$};
\node[blue]  at (54:1.1) {\small $p_5$};
    \end{tikzpicture}
 \end{center}
 Following Theorem 2.30 in \cite{SWflag}, the space $\mathscr{P}_{\G, D_n}$ is a smooth affine variety. The cyclic group $\mathbb{Z}/n$ acts on $\mathscr{P}_{\G, D_n}$ by rotation. 
\end{example}

\subsection{Cluster Poisson structures} 
As shown in \cite{GS3}, there are four groups acting on $\mathscr{P}_{\G, \bS}$:
\begin{itemize}
    \item the \emph{mapping class group} of $\bS$,
    \item the group of \emph{outer automorphisms} of $\G$,
    \item the product of  \emph{Weyl groups} over  punctures of $\bS$,
    \item the product of \emph{braid groups} over boundary circles of $\bS$. 
\end{itemize}

\bt[{\cite[Theorem 1.12]{GS3}}] The space $\mathscr{P}_{\G, \bS}$ admits a natural cluster Poisson structure.  The above four groups act on  $\mathscr{P}_{\G, \bS}$ as quasi-cluster Poisson transformations.
\et

Below we briefly recall the cluster Poisson structure on $\mathscr{P}_{\G, \bS}$. 

Let us focus on the case when $\bS$ is a triangle $t$. 
The Poisson structure on $\mathscr{P}_{\G, t}$ can be described  by quivers. When $\G={\rm PGL}_{r+1}$, its corresponding quiver $Q$ is constructed by Fock and Goncharov \cite{FGteich}, and is shown on Figure \ref{dnfonqasacxs21}. 
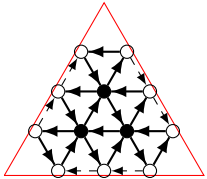
\begin{figure}[b]
\begin{center}
\begin{tikzpicture}[scale=0.61]
\draw[thin, red] (0:0) -- (0:4.35) -- (60:4.35) --(0:0);
\begin{scope}[shift={(30:0.2)}]
\foreach \n in {1,2,3}
{\node[circle,draw,minimum size=5pt,inner sep=0pt] (\n) at (60:\n) {};
\node[circle,draw,minimum size=5pt,inner sep=0pt] (b\n) at (0:\n) {};
\node[circle,draw,minimum size=5pt,inner sep=0pt] (c\n) at ([shift={(4,0)}]120:\n) {};
}
\node[circle,fill,draw,minimum size=5pt,inner sep=0pt] (A) at ([shift={(1,0)}]60:1) {};
\node[circle,fill,draw,minimum size=5pt,inner sep=0pt] (B) at ([shift={(1,0)}]60:2) {};
\node[circle,fill,draw,minimum size=5pt,inner sep=0pt] (C) at ([shift={(2,0)}]60:1) {};
 \draw[directed,dashed] (1) -- (2); 
  \draw[directed,dashed] (2) -- (3); 
   \draw[directed,dashed] (b3) -- (b2); 
  \draw[directed,dashed] (b2) -- (b1); 
 \draw[directed,dashed] (c3) -- (c2); 
  \draw[directed,dashed] (c2) -- (c1); 
\foreach \from/\to in {c1/C,C/A, A/1, c2/B,B/2,c3/3,1/b1, 2/A, A/b2, 3/B,B/C,C/b3, b3/c1, C/c2, b2/C, b1/A, A/B, B/c3}
                 \draw[directed, thick] (\from) -- (\to);    
\end{scope}                 
\end{tikzpicture}
\end{center}
\caption{The quiver for $\mathscr{P}_{{\rm PGL}_4, t}$. The  bold vertices are mutable, the  hollow vertices on the sides are frozen, and the dashed arrows between hallow vertices are of weight $\frac{1}{2}$.}
\label{dnfonqasacxs21}
\end{figure}
The number of vertices of $Q$ is 
 \[
 n=\dim \mathscr{P}_{{\rm PGL}_{r+1}, t}= \dim \B + {\rm rk} \G = \frac{(r+5)r}{2}.
 \]
The quiver $Q$ encodes a skew-symmetric $n\times n$ matrix $\widehat{\varepsilon}=(\widehat{\varepsilon}_{ij})$, where 
\[
\widehat{\varepsilon}_{ij} =\#\{ \mbox{arrows from $i$ to $j$}\} - \#\{ \mbox{arrows from $j$ to $i$}\}.
\]
 The space $\mathscr{P}_{{\rm PGL}_{r+1}, t}$ has a cluster Poisson chart ${\bf x}$ whose coordinates $x_1, \ldots, x_n$ correspond to the vertices of $Q$.
 The coordinates $x_i$ corresponding to the bold vertices in Figure \ref{dnfonqasacxs21} are mutable and are defined by taking  the triple ratios of certain derived flags (\cite[\S 9.3]{FGteich}). The group ${\rm PGL}_{r+1}$ is simply-laced and the multipliers $d_i$ are taken to be 1. Putting them together, we obtain a seed $\Sigma=({\bf x}, \varepsilon)$ for $\mathscr{P}_{{\rm PGL}_{r+1}, t}$. 

For an arbitrary semisimple group $\G$, the quivers and cluster Poisson coordinates for $\mathscr{P}_{\G, t}$ are systematically constructed in Section 5.3 of \cite{GS3}. Let  $v$ be a vertex of $t$ and let ${\bf i}$ be a reduced decomposition of the longest Weyl group element $w_0$. For every pair $(v, {\bf i})$, the paper \cite{GS3} constructed a quiver $Q_{v, {\bf i}}$ for $\mathscr{P}_{\G, t}$. The construction of $Q_{v, {\bf i}}$ generalizes the amalgamation procedure developed in \cite{FGamalgamation},  Here are several features of $Q_{v, {\bf i}}$.
\begin{itemize}
\item The mutable part of the quiver $Q_{v, {\bf i}}$ coincides with the mutable part of the quiver for the double Bruhat cell $\G^{e, w_0}:= \B \cap \B^-w_0\B^-$ in \cite{BFZ}. 
\item The frozen part of the quiver $Q_{v, {\bf i}}$ contains  $3r$ vertices, where $r={\rm rk}\G$.
\item In general the quiver $Q_{v, {\bf i}}$ is not planar. However one may still place $Q_{v, {\bf i}}$ on the top of $t$ such that each side of $t$ contains $r$ frozen vertices. 
\end{itemize}
The paper \cite{GS3} further assigned   a cluster Poisson chart ${\bf x}_{v, {\bf i}}$ to every pair $(v,{\bf i})$ and proved that the seeds $\Sigma_{v, {\bf i}}=({\bf x}_{v, {\bf i}}, \varepsilon_{v, {\bf i}})$ are  mutation equivalent. 

Now let $\bS$ be a general surface. An {ideal triangulation} $T$ of $\bS$ is a triangulation of $\bS$ whose vertices are the punctures and marked points of $\bS$. Consider the {\it decorated triangulation}
\be
\label{dec.ideal.tri}
\widetilde{T}:=(T, \{v_t\}, \{{\bf i}_t\}),
\ee
where $T$ is an ideal triangulation,  ${\bf i}_t$ is a reduced decomposition of $w_0$ assigned to every triangle $t\in T$, and $v_t$ is a vertex of $t$. By amalgamating the local seeds $\Sigma_{v_t, {\bf i}_t}$ along the frozen vertices on corresponding edges of $t$'s, with $t$ running through the triangles of $T$, we obtain a seed $\Sigma_{\widetilde{T}}$ for $\mathscr{P}_{\G, \bS}$. 
By Definition \ref{svafonv}, we consider the cluster Poisson algebra 
\be
\label{sdafionio}
\mathcal{O}_{\rm cl}(\mathscr{P}_{\G, \bS}):=\mathcal{O}(\Sigma_{\widetilde{T}}).
\ee
Theorem 5.11 of \cite{GS3} states that the seeds $\Sigma_{\widetilde{T}}$ associated to all the decorated triangulation $\widetilde{T}$ are mutation equivalent. Therefore $\mathcal{O}_{\rm cl}(\mathscr{P}_{\G, \bS})$ is independent of the  $\widetilde{T}$ chosen.

\section{Proofs}
\subsection{Proof of Theorem \ref{cdsn.cluster.poisson}}
\label{sec.3.3}
The proof of Theorem \ref{cdsn.cluster.poisson} uses a standard codimension 2 argument as in \cite{BFZ}.

\vskip 1mm

When $\bS=D_n$ is a polygon, by definition we have 
\[
\mathscr{P}_{\G, D_n}\stackrel{\sim}{=} {\rm Conf}_{w_0^{n-2}}^e(\mathcal{A}_{\rm ad}) \times {\rm H}^{n-2},
 \]
 where ${\rm Conf}_{w_0^{n-2}}^e(\mathcal{A}_{\rm ad})$ is a double Bott-Samelson variety in \cite{SWflag}. In this case, the proof  goes through the same steps as in the proof of  Theorem 3.44 of \cite{SWflag}.
In fact, as we will see soon, it suffices to first prove the cases when $n=3, 4$, which correspond to the double Bruhat cells $\G^{e, w_0}$ and $\G^{w_0, w_0}$; the theorem follows easily from Theorem 2.10 of \cite{BFZ}.

\vskip 1mm

Let  $\widetilde{T}=(T, \{v_t\}, \{{\bf i}_t\})$ be a decorated triangulation of an arbitrary surface $\bS$. The edges of the ideal triangulation $T$ that are not boundary intervals of $\bS$ are called {\it diagonals}.

For an arbitrary diagonal $e$, let $\Sigma_{\widetilde{T}, e}$ be the seed obtained from $\Sigma_{\widetilde{T}}$ by freezing all the mutable vertices that are placed on the diagonals different from $e$.

Let $\mathscr{P}_{\G, \bS}^{T, e} \subset \mathscr{P}_{\G, \bS}$ consist of data $(\mathcal{L}, \{\B_c\}, \{\B_i\}, \{p_e\})$ such that 
 for every diagonal $f\neq e$ in $T$, the flags $(\B_f^+, \B_f^-)$ sitting at the ends of $f$ are generic.

\begin{lemma}
\label{sadnoin}
The ring  $\mathcal{O}(\mathscr{P}_{\G, \bS}^{T,e})$ coincides with the cluster Poisson algebra $\mathcal{O}(\Sigma_{\widetilde{T},e})$. 
\end{lemma}

\begin{proof}

Let $q$ be the quadrilateral in $T$ that contains $e$ as a diagonal.
As illustrated below, the surface $\bS$ can be obtained by gluing  $q$ and the triangles of $T$ that  do not contain $e$. 
\begin{center}
\begin{tikzpicture}[scale=0.6]
\draw[thin] (0,0) -- (0,-4) -- (4,-4) --(4,0)--(0,0);
\draw[dashed] (0,0)--(1.5,-1.5)--(0,-4);
\draw[dashed] (1.5,-1.5)--(4,0);
\node[blue] at (1.5,-1.5) {$\circ$};
\node[red] at (0,0) {$\tiny \bullet$};
\node[red] at (0,-4) {$\tiny \bullet$};
\node[red] at (4,-4) {$\tiny \bullet$};
\node[red] at (4,0) {$\tiny \bullet$};
\node at (1.5, -.5) {$t_1$};
\node at (.5, -1.5) {$t_2$};
\node at (2.5,-2.5) {$q$};

\begin{scope}[shift={(-10,0.5)}]
\node at (1.5, -.5) {$t_1$};
\draw[thin] (0,0) -- (1.5,-1.5) --(4,0)--(0,0);
\node[red] at (0,0) {$\tiny \bullet$};
\node[red] at (1.5,-1.5) {$\tiny \bullet$};
\node[red] at (4,0) {$\tiny \bullet$};
\end{scope}    

\begin{scope}[shift={(-10.5,0)}]
\node at (.5, -1.5) {$t_2$};
\draw[thin] (0,0) -- (1.5,-1.5) --(0,-4)--(0,0);
\node[red] at (0,0) {$\tiny \bullet$};
\node[red] at (1.5,-1.5) {$\tiny \bullet$};
\node[red] at (0,-4) {$\tiny \bullet$};
\end{scope}  

\begin{scope}[shift={(-9.5,-.5)}]
\node at (2.5,-2.5) {$q$};
\draw[thin] (4,0) -- (1.5,-1.5) --(0,-4)--(4,-4)--(4,0);
\node[red] at (4,0) {$\tiny \bullet$};
\node[red] at (1.5,-1.5) {$\tiny \bullet$};
\node[red] at (0,-4) {$\tiny \bullet$};
\node[red] at (4,-4) {$\tiny \bullet$};
\end{scope}  

\draw[-latex, thick] (-4,-2) -- (-2,-2);
\end{tikzpicture}
\end{center}
By \cite[Lemma 2.12]{GS3}, there is a gluing map 
\[
\gamma:~\left(\prod_{t\in T ~|~ e\notin t} \mathscr{P}_{\G, t}\right)\times \mathscr{P}_{\G, q} \longrightarrow \mathscr{P}_{\G, \bS},
\]
which identifies the pinnings on the corresponding pairs of glued edges. By definition, the image of $\gamma$ is the subspace $\mathscr{P}_{\G, \bS}^{T, e}$. The pull-back of $\gamma$ is an injective algebra homomorphism 
\be
\label{sadv}
\gamma^\ast: ~\mathcal{O}(\mathscr{P}_{\G, \bS}^{T, e}) ~{\hlra} ~\left( \bigotimes_{t\in T ~|~ e\notin t} \mathcal{O}(\mathscr{P}_{\G, t})\right) \bigotimes \mathcal{O}(\mathscr{P}_{\G, q}).
\ee
Recall the procedure of amalgamation. The pull-back of $\gamma$ induces an injection of cluster Poisson algebras
\be
\label{acsniadv20}
\gamma^\ast: ~ \mathcal{O}(\Sigma_{\widetilde{T}, e}) ~{\hlra} ~  \left( \bigotimes_{t\in T ~|~ e\notin t} \mathcal{O}_{\rm cl}(\mathscr{P}_{\G, t})\right) \bigotimes \mathcal{O}_{\rm cl}(\mathscr{P}_{\G, q}).
\ee
Here if $i$ is a vertex on a diagonal $f\neq e$, then $\gamma^\ast(x_i)$ is the product $x_i'x_i''$ of variables on corresponding edges of the righthand side of \eqref{acsniadv20}. Otherwise $\gamma^\ast(x_i)$ remains invariant.

 By the above discussion on polygon cases, we get
 \[
 \left( \bigotimes_{t\in T ~|~ e\notin t} \mathcal{O}(\mathscr{P}_{\G, t})\right) \bigotimes \mathcal{O}(\mathscr{P}_{\G, q}) =  \left( \bigotimes_{t\in T ~|~ e\notin t} \mathcal{O}_{\rm cl}(\mathscr{P}_{\G, t})\right) \bigotimes \mathcal{O}_{\rm cl}(\mathscr{P}_{\G, q}).
 \]
 Let $F$ be a rational function of $\mathscr{P}_{\G, \bS}^{T, e}$. We have
\begin{align*}
 F \in  \mathcal{O}(\mathscr{P}_{\G, \bS}^{T, e}) 
~~ &\Longleftrightarrow ~~ \gamma^\ast(F) \in \left( \bigotimes_{t\in T ~|~ e\notin t} \mathcal{O}(\mathscr{P}_{\G, t})\right) \bigotimes \mathcal{O}(\mathscr{P}_{\G, q}) 
 \\
 & \Longleftrightarrow ~~ \gamma^\ast(F) \in  \left( \bigotimes_{t\in T ~|~ e\notin t} \mathcal{O}_{\rm cl}(\mathscr{P}_{\G, t})\right) \bigotimes \mathcal{O}_{\rm cl}(\mathscr{P}_{\G, q})\\
 & \Longleftrightarrow ~~ F\in \mathcal{O}(\Sigma_{\widetilde{T}, e}).
\end{align*}
The Lemma is proved.
\end{proof}
Now let us go through all the diagonals $e$ of $T$ and set
\[
\mathscr{P}_{\G, \bS}':=\bigcup_{e}\mathscr{P}_{\G, \bS}^{T, e}~\subset ~\mathscr{P}_{\G, \bS}. 
\]
By Lemma \ref{GHK.up} and Lemma \ref{sadnoin}, we get 
\[
\mathcal{O}\left( \mathscr{P}_{\G, \bS}' \right)= \bigcap_{e} \mathcal{O}(\mathscr{P}_{\G, \bS}^{T, e})=\bigcap_e \mathcal{O}(\Sigma_{\widetilde{T},e}) =\bigcap_e \mathcal{U}(\Sigma_{\widetilde{T},e}) = \mathcal{U}(\Sigma_{\widetilde{T}})= \mathcal{O}_{\rm cl}(\mathscr{P}_{\G, \bS}) .
\]
For a pair $e, f$ of diagonals of $T$, let $\mathscr{P}_{\G, \bS}^{e, f}\subset \mathscr{P}_{\G, \bS}$ be  the subspace such that the pairs of flags associated to $e$ and $f$ are not generic. Clearly the spaces $\mathscr{P}_{\G,\bS}^{e,f}$ are of codimension $\geq 2$ in $\mathscr{P}_{\G, \bS}$.  By definition, the compliment of $\mathscr{P}_{\G, \bS}'$ in $\mathscr{P}_{\G, \bS}$ is the union $\bigcup_{e,f} \mathscr{P}_{\G,\bS}^{e,f}.$
Therefore $\mathcal{O}({\mathscr{P}_{\G, \bS}})=\mathcal{O}({\mathscr{P}_{\G, \bS}'})$, which concludes the proof of Theorem \ref{cdsn.cluster.poisson}. 

\subsection{Grothendieck-Springer resolution}
We consider the \emph{Grothendieck-Springer} resolution 
\[\widetilde{\G}:=\left\{(g, \B)\in \G\times \mathcal{B} ~|~ g\in \B\right\}.
\]
The second projection $\pi_2$ from $\widetilde{\G}$ to $\mathcal{B}$ makes $\widetilde{\G}$ a smooth $\B$-bundle over $\mathcal{B}$. The first projection $\pi_1$ from  $\widetilde{\G}$ to $\G$ generically is  $|W|$-to-1. The Weyl group birationally acts on $\widetilde{\G}$ by altering the flag $\B$. We include a careful proof of the following  result. 
\bp 
\label{qncfdo}
The action of the Weyl group $W$ on $\widetilde{\G}$ induces automorphisms of the ring $\mathcal{O}(\widetilde{\G})$. The first projection $\pi_1$ gives a canonical algebra isomorphism
\[
\pi_1^\ast: ~\mathcal{O}(\G)\stackrel{\sim}{\lra} \mathcal{O}(\widetilde{\G})^W. 
\]
\ep

\begin{proof}
 Recall the Cartan subgroup $\H$. For every Borel subgroup $\B\in \mathcal{B}$, there is a canonical surjective homomorphism
\be
\label{pqrnjamdp}
\pi_{\B}: \B\lra \B/[\B, \B]\stackrel{\sim}{\longrightarrow} {\rm H}. 
\ee
We obtain a morphism 
\[
\eta: ~ \widetilde{\G}\longrightarrow \G\times {\rm H},\hskip 7mm (g, \B)\longmapsto (g, \pi_\B(g)).
\]
Denote by $\widetilde{\G}_{\bf a}$ the image of $\eta$. Note that $\widetilde{\G}_{\bf a}$ is affine but not smooth in general. For example, when $\G={\rm SL}_2$, then 
\[
\widetilde{\G}_{\bf a}=\left\{(g,h)\in {\rm G}\times {\rm H} ~\middle |~ {\rm tr}(g) ={\rm tr}(h)\right\}. 
\]
In this case, $\widetilde{\G}_{\bf a}$ has two singular points when $g=h=\pm I_2$.

We show that $\widetilde{\G}_{\bf a}$ is an affinization of $\widetilde{\G}$. For every $g\in\G$, the dimension of its centralizer  $Z(g)$ is greater or equal to the rank of $\G $. The element  $g$ is defined to be \emph{regular} if 
\[\dim Z(g)={\rm rk} \G.\] 
For instance, when $\G={\rm SL}_n$, an element $g$ is regular  if and only if its Jordan normal form contains a single Jordan block for each eigenvalue. See \cite{Ste} for further expositions. 

Let us set 
\[
Q=\left\{(g, \B)\in \widetilde{\G}~\middle|~ \mbox{$g$ is irregular}\right\}.
\]
For every $\B\in \mathcal{B}$, the subvariety $Q\cap \pi_2^{-1}(\B)$ consists of irregular elements in $\B$ and  is of codimension 2 in the fiber $\pi_2^{-1}(\B)$. 
Therefore $Q$ is of codimension 2 in $\widetilde{\G}$.
By \cite[Thm.1.1 \& 1.2]{Ste}, the projection $\eta$ from $\widetilde{\G}$ to $\widetilde{\G}_{\bf a}$ is an ismorphism outside $Q$.
 Therefore we get
\be
\label{vfadpmn}
\mathcal{O}(\widetilde{\G}){=}\mathcal{O}(\widetilde{\G}_{\bf a}).
\ee
The $W$-action on ${\rm H}$ extends to a regular $W$-action on $\widetilde{\G}_{\bf a}$. By definition, the projection $\eta$ from $\widetilde{\G}$ to $\widetilde{\G}_{\bf a}$ is $W$-equivarient. Therefore $W$ maps regular functions of $\widetilde{\G}$ to regular functions, which complete the proof of the first claim.

The fibers of the projection from $\widetilde{\G}_{\bf a}$ to $\G$ are $W$-orbits. Hence the the pull-back of the projection give an injective homomorphism 
\be
\la{sdnqcvio}
\mathcal{O}(\G) \hlra \mathcal{O}(\widetilde{\G}_{\bf a})^W.
\ee
It remains to show the surjectivity. Let $F\in \mathcal{O}(\widetilde{\G}_{\bf a})^W$. Since $\widetilde{\G}_{\bf a}$ is an affine subvariety of $\G \times {\rm H}$, the function $F$ can be obtain as the restriction of a function 
\[\hat{F}\in \mathcal{O}(\G \times {\rm H})=\mathcal{O}(\G)\otimes \mathcal{O}({\rm H}).\] 
Let us write
\[
F(g,h)=\hat{F}(g,h)= \sum_{i=1}^s P_i(g)Q_i(h).
\]
Without loss of generality, let us assume that the functions $Q_i$ are $W$-invariant. Otherwise, one may consider the $W$-orbit containing $\hat{F}$ and then take the average. By the Chavelley restriction theorem, there is a canonical isomorphism
\[
 \mathcal{O}({\rm H})^W = \mathcal{O}(\G)^{\G}.
\]
In particular, every function $Q_i\in \mathcal{O}({\rm H})^W$ uniquely corresponds to a  function $Q_i'\in \mathcal{O}(\G)^{\G}$ such that whenever $(g, h)\in \widetilde{\G}_{\bf a}$ one has
\[
Q_i(h)=Q_i'(g).
\]
Therefore $F$ can be obtained from a function $\sum_{i=1}^r P_i Q_i'\in \mathcal{O}(\G)$ and  \eqref{sdnqcvio} is surjective.  
\end{proof}

Let $\bS=\odot$ be a once punctured disk with two marked points on its boundary. Its fundamental group $\pi_1(\odot)=\Z$. Local systems over $\odot$ are determined by  the monodromies surrounding the puncture. The space $\mathscr{P}_{\G, \odot}$ parametrizes $\G$-orbits of  the data $(g, \B, \B_1, \B_2, p_1, p_2)$ as follows.
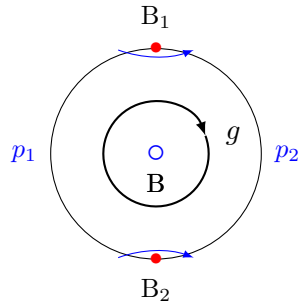
\begin{center}
\begin{tikzpicture}[scale=0.7]
\draw  (0,0) circle (20mm);
\node [blue, thick]  at (0,0) {\Large $\circ$};
\node [red]  at (0,2) {\small $\bullet$};
\node [red]  at (0,-2) {\small $\bullet$};
\node [label=above: {\small $\B_1$}] at (0,2) {};
\node [label=below: {\small $\B_2$}] at (0,-2) {};
\node [label=below: {\small $\B$}] at (0,0) {};
\node [blue] at (-2.5,0) {\small $p_1$ };
\node [blue] at (2.5,0) {\small $p_2$ };
\node [label=below: {\small $\B$}] at (0,0) {};
\draw[thick, -latex ] ([shift=(20:10mm)]0,0) arc (380:20:10mm);
\node [label=right: {$g$}] at (20:1) {};
\draw[blue, latex- ] ([shift=(70:21mm)]0,0) arc (-70:-110:21mm);
\draw[blue, latex- ] ([shift=(-70:21mm)]0,0) arc (70:110:21mm);
    \end{tikzpicture}
 \end{center}
Here the group element $g$ is contained in the Borel subgroup $\B$ assigned to the puncture. If $g$ is regular and semisimple, then the set of Borel subgroups containing $g$ is a Weyl group torsor. The Weyl group acts birationally on $\mathscr{P}_{\G, \odot}$ by altering  $\B$. 

Fix a pinning $p=(\B_+, \B_-, x_i, y_i; i\in I)$.
The next Lemma gives an alternative description of $\mathscr{P}_{\G, \odot}$. As an easy corollary, $\mathscr{P}_{\G, \odot}$ is a smooth variety but not  affine. 
\begin{lemma}
\label{decription.p0dot}
The space $\mathscr{P}_{\G, \odot}$ is naturally isomorphic to a subvariety of $\B_+\times \B_-\times \mathcal{B}$:
\be
\label{dnws}
\mathscr{P}_{\G, \odot}\stackrel{\sim}{=}\left \{ (b_1, b_2, \B) \in \B^+\times\B^-\times \mathcal{B} ~\middle|~ b_1b_2^{-1}\in \B \right\}.
\ee
\end{lemma}
\begin{proof}
 Every element in $\mathscr{P}_{\G, \odot}$ contains a unique representative with its left pinning $p_1=p$. Along the path near $\B_1$ (illustrated by the blue arrow in the above figure), there is a unique element $b_1\in \B^+$ transporting $p_1$ to $p_2$. Similarly, there is an element $b_2\in \B^-$ along the path near $\B_2$ transporting $p_1$ to $p_2$. The monodromy $g=b_1b_2^{-1}\in \B$. In this way, we get a morphism from $\mathscr{P}_{\G, \odot}$ to the righthand side of \eqref{dnws}. It is clear that the morphism is invertible. 
\end{proof}

By forgetting the flags assigned to the puncture of $\odot$, we get a natural projection
\be
\label{adondndwvn}
\pi:~ \mathscr{P}_{\G, \odot} \lra \B^+\times \B^-.
\ee
\begin{lemma} The birational Weyl group action on $\mathscr{P}_{\G, \odot}$ induces automorphisms of the coordinate ring $\mathcal{O}(\mathscr{P}_{\G, \odot})$. The pull-back of \eqref{adondndwvn} gives an isomorphism:
\[
\pi^\ast: ~\mathcal{O}(\B^+\times \B^-)\stackrel{\sim}{\lra} \mathcal{O}\left(\mathscr{P}_{\G, \odot}\right)^W.
\]
\end{lemma}
\begin{proof}
Take the open cell $\G_0:= \B^+ \B^-$ inside $\G$. Let $\widetilde{\G}_0=\pi_1^{-1}(\G_0)$ be its preimage inside the Grothendieck-Springer resolution $\widetilde{\G}$.
We restrict the action of the Weyl group to $\widetilde{\G}_0$.
Going along the same lines as the proof of Proposition \ref{qncfdo}, we get an isomorphism
\[
\mathcal{O}(\G_0) \stackrel{\sim}{\longrightarrow} \mathcal{O}(\widetilde{\G}_0)^W.
\]
There is an isomorphism from $\B^+\times \B^-$ to $\G_0\times {\rm H}$ that maps $(b_1, b_2)$ to $(b_1b_2^{-1}, \pi_{\B_+}(b_1))$. Similarly, there is an isomorphism from $\mathscr{P}_{\G, \odot}$ to $\widetilde{\G}_0\times {\rm H}$.
Putting them together, we obtain the following commutative diagram:
\[\begin{tikzcd}
\mathscr{P}_{\G, \odot} \arrow{r}{\sim} \arrow[swap]{d}{\pi} & \widetilde{\G}_0\times {\rm H} \arrow{d}{\pi_1 \times {id}} \\
\B^+\times \B^-\arrow{r}{\sim}& \G_0 \times {\rm H}.
\end{tikzcd}
\]
The Weyl group action on $\mathscr{P}_{\G, \odot}$ becomes the Weyl group action on $\widetilde{\G}_0$ plus the trivial action on ${\rm H}$. Therefore,
\[
\pi^*: ~\mathcal{O}(\B^+\times \B^-) \stackrel{\sim}{=} \mathcal{O}(\G_0) \otimes \mathcal{O}({\rm H}) \stackrel{\sim}{\lra}  \mathcal{O}(\widetilde{\G}_0)^W \otimes \mathcal{O}({\rm H}) \stackrel{\sim}{=} \mathcal{O}(\mathscr{P}_{\G, \odot})^W. \qedhere
\]
\end{proof}

Define the {\it outer monodromy map} 
\be
\label{qfnonwrqp}
\tau: \B^+\times \B^- \lra {\rm H}, \hskip 7mm (b_1, b_2) \lms \pi_{\B^+}(b_1)\pi_{\B^-}(b_2).
\ee
Following \eqref{cwnionwv}, the dual Poisson-Lie group $\G^\ast$ is a subgroup of $\B^+\times \B^-$, given by the fiber over $e$ of the outer monodromy map $\tau$. 
Let $\alpha_1, \ldots, \alpha_r$ be simple positive roots of $\G$. The {\it outer monodromy} functions are 
\[
\mu_i: ~ \mathscr{P}_{\G, \odot}\stackrel{\pi}{\lra}\B^+\times \B^-\stackrel{\tau}{\lra}{\rm H}\stackrel{\alpha_i}{\lra} \mathbb{C}.
\]
Let $\mathcal{I}$ be an ideal of $\mathcal{O}( \mathscr{P}_{\G, \odot})^W$ generated by $\mu_i-1$ for $1\leq i\leq r$. Then we get an isomorphism
\be
\label{sfivqn}
\iota:\, \mathcal{O}(\mathscr{P}_{\G, \odot})^W {\Big \slash} \mathcal{I} \stackrel{\sim}{\lra} \mathcal{O}(\G^\ast).
\ee

\subsection{Braid group action and  Proof of Theorem \ref{main.result.1}}
\label{braid.group.action.4}
Let $C=(C_{ij})$ be the Cartan matrix associated to $\G$. For any $i\neq j \in I$, we set $m_{ij}=2,3,4$ or $6$ according to whether $C_{ij}C_{ji}$ is $0, 1, 2$ or $3$. The braid group $\mathbb{B}_\G$ is generated by $\sigma_i$ ($i \in I$), and satisfies the relations
\[
\sigma_i\sigma_j\sigma_i\ldots = \sigma_j \sigma_i\sigma_j \ldots,
\]
with both sides having $m_{ij}$ factors.

The paper \cite{GS3} presents a geometric braid group action on $\mathscr{P}_{\G, \odot}$ and proves its cluster nature. 
Under the projection \eqref{adondndwvn}, the braid group action on $\mathscr{P}_{\G, \odot}$ descends to a braid group action on $\B^+\times \B^-$ and further on $\G^\ast$ by restriction. 
By Theorem 2.29(4) of \cite{GS3}, when $\G$ is simply laced, the cluster quantization of this braid group action on $\G^\ast$ recovers  Lusztig's braid group actions on the quantum group $\mathcal{U}_q(\mathfrak{g})$.   

Below we present an explicit description of the induced braid group action on $\B^+\times \B^-$.

Let $\U^{+}=[\B^+, \B^+]$ and $\U^-=[\B^-, \B^-]$ be the maximal unipotent subgroups inside $\B^+$ and $\B^-$ respectively. 
Every $g\in \U^+{\rm H}\U^-$ admits a unique decomposition $g=[g]_+[g]_0[g]_-$ such that $[g]_+\in \U^+$, $[g]_0\in {\rm H}$, and $[g]_-\in {\rm U}^-$.
The standard pinning $p=(\B^+, \B^-, x_i, y_i; i\in I)$ determines a homomorphism $\gamma_i$ from ${\rm SL}_2$ to $\G$ for  $i\in I$. 
Correspondingly, there are additive characters $\chi_i: \U^+\rightarrow \C$ and $\chi_i^-: \U^-\rightarrow \C$ such that
\be
\label{sdoowdcvbo}
\chi_i(x_j(a))=\delta_{ij} a, \hskip 7mm \chi_i^-(y_j(b))=\delta_{ij}b,
\ee
where $\delta_{ij}$ is the Kronecker delta.

For $i\in I$, we get a Poisson automorphism 
\begin{align*}
\sigma_i: \B^+\times \B^- &\lra \B^+\times \B^-\\
(b_1, b_2) &\longmapsto (t_1 b_1t_2, t_1b_2t_2).
\end{align*} where
\[
t_1= \gamma_i \begin{pmatrix} 0 & ~1\\
-1 & ~\chi_i([b_1]_+)\\\end{pmatrix} ,  \hskip 10mm t_2= \gamma_i \begin{pmatrix} 0 & -1\\
1 & ~\chi_i^-([b_2]_-)\\\end{pmatrix}.
\]
The maps $\sigma_i$ satisfy the braid relation and therefore gives rise to a $\mathbb{B}_\G$-action on $\B^+\times \B^-$. 
The outer monodromy map $\tau$ in \eqref{qfnonwrqp} intertwines the braid group action on $\B^+\times \B^-$ with the Weyl group action on ${\rm H}$. Therefore it induces a $\mathbb{B}_\G$-action on the fiber $\G^\ast:= \tau^{-1}(e)$.

\begin{example} 
\label{adjoncaoda}
The group ${\rm PGL}_3^\ast$ consists of elements $(u_1h, h^{-1}u_2) \in \B_+\times \B_-$, where
\[
u_1=\begin{pmatrix} 1 & ~e_1 & ~e_3\\
0 & 1 & e_2\\
0 & 0 & 1\\\end{pmatrix} ,  \hskip 10mm 
h = \begin{pmatrix} \kappa_1\kappa_2 & ~0 & ~0\\
0 & \kappa_2 & 0\\
0 & 0 & 1 \\\end{pmatrix} ,
 \hskip 10mm 
u_2 = \begin{pmatrix} 1 & ~0 & ~0\\
f_1 & 1 & 0\\
f_3 & f_2 & 1\\\end{pmatrix}.
\]
The map $\sigma_1$ takes the element $(u_1h, h^{-1}u_2)$  to $(u_1'h', h'^{-1} u_2')$, where
\[
u_1'=\begin{pmatrix} 1 & ~\kappa_1^{-1}f_1 & ~e_2\\
0 & 1 & e_1e_2-e_3\\
0 & 0 & 1\\\end{pmatrix} ,  \hskip 10mm 
h' = \begin{pmatrix} \kappa_2 & ~0 & ~0\\
0 & \kappa_1\kappa_2 & 0\\
0 & 0 & 1\\\end{pmatrix} ,
 \hskip 10mm 
u_2' = \begin{pmatrix} 1 & ~0 & ~0\\
e_1\kappa_1 & 1 & 0\\
f_2 & f_1f_2-f_3 & 1\\\end{pmatrix}.
\]
The map $\sigma_2$ takes the element $(u_1h, h^{-1}u_2)$  to $(u_1''h'', h''^{-1} u_2'')$, where
\[
u_1''=\begin{pmatrix} 1 & ~e_3 & ~e_3\kappa_2^{-1}f_2 -e_1\\
0 & 1 & \kappa_2^{-1}f_2\\
0 & 0 & 1\\\end{pmatrix} ,  \hskip 10mm 
h'' = \begin{pmatrix} \kappa_1 & ~0 & ~0\\
0 & \kappa_2^{-1} & 0\\
0 & 0 & 1\\\end{pmatrix} ,
 \hskip 10mm 
u_2'' = \begin{pmatrix} 1 & ~0 & ~0\\
f_3 & 1 & 0\\
e_2\kappa_2f_3-f_1 & e_2\kappa_2 & 1\\\end{pmatrix}.
\]
\end{example}

Now we prove Theorem \ref{main.result.1}. Let $\mathcal{O}_q(\mathscr{P}_{\G, \odot})$ be the quantized algebra as in \eqref{evfoqscsvnj}. The Weyl group action are cluster transformations and therefore can be lift to an action on $\mathcal{O}_q(\mathscr{P}_{\G, \odot})$. The functions $\mu_i$ belong to the Poisson center of $\mathcal{O}(\mathscr{P}_{\G, \odot})$ and are expressed as  Laurent monomials in every seed of $\mathcal{O}(\mathscr{P}_{\G, \odot})$. Therefore, by the quantum Lift Theorem in \cite[Appendix.A3]{GS3}, the functions $\mu_i$ can be uniquely lift to the central elements in $\mathcal{O}_q(\mathscr{P}_{\G, \odot})^W$. We still denote by $\mathcal{I}$ the ideal of $\mathcal{O}_q( \mathscr{P}_{\G, \odot})^W$ generated by $\mu_i-1$.

Theorem 2.22 of \cite{GS3} shows that there is a canonical map of $\ast$-algebra:
\be
\label{anxsicnaidhbas}
\kappa: \mathcal{U}_q(\mathfrak{g}) \longrightarrow \mathcal{O}_q(\mathscr{P}_{\G,\bS})^W{\Big \slash} \mathcal{I}.
\ee
Taking the semiclassical limit of $\kappa$, we get a Poisson-algebra homomorphism 
\be
\label{anxsicnaidhb}
\kappa: \mathcal{O}(\G^\ast) \longrightarrow \mathcal{O}_1(\mathscr{P}_{\G, \odot})^W{\Big \slash} \mathcal{I} \stackrel{(\ast)}{\longrightarrow} \mathcal{O}(\mathscr{P}_{\G, \odot})^W {\Big \slash} \mathcal{I},
\ee
where the second map $(\ast)$ is a special case of  the homomorphism \eqref{sacfdouncdfoq}. 
By the following Lemma,  the map \eqref{anxsicnaidhb} is an isomorphism, which concludes the proof of Theorem \ref{main.result.1}.
\begin{lemma} The map $\kappa$ in \eqref{anxsicnaidhb} is the inverse map of the isomorphism $\iota$ in \eqref{sfivqn}.
\end{lemma}
\begin{proof} We shall adopt the notations used in \cite{GS3}. Following  \cite[\S 11.1]{GS3}, the quantum group $\mathcal{U}_q(\mathfrak{g})$ is generated by the rescaled generators $\{{\bf E}_i, {\bf F}_i, {\bf K}_i^{\pm 1}\}_{i\in I}$ satisfying relations (398)-(399) therein. 
By Theorem 7.6 of \cite{CKP}, after taking the semiclassical limit, the above generators become the following regular functions in $\mathcal{O}(\G^\ast)$:
\[
{\bf E}_i \lms \chi_i(u_1), \hskip 10mm {\bf F}_i \lms \chi_i^-(u_2), \hskip 10mm {\bf K}_i \lms \alpha_i(h), \hskip 10mm \forall \, (u_1h, h^{-1}u_2) \in \G^\ast.
\]
where  $\chi_i$ and $\chi_i^-$ are as in \eqref{sdoowdcvbo}, and $\alpha_i$ is the $i$th simple positive root. 
By Theorem 2.22 of \cite{GS3}, under the map \eqref{anxsicnaidhbas}, we get the following quantized functions in $\mathcal{O}_q(\mathscr{P}_{\G,\bS})^W{\Big \slash} \mathcal{I}$.
\[
\kappa({\bf E}_i)={\rm W}_{e, i}, \hskip 10mm \kappa({\bf F}_i)={\rm W}_{f,i^\ast}, \hskip 7mm \kappa({\bf K}_i) =\rho^\ast_e \alpha_i. 
\]
By definition, the semiclassical limits of ${\rm W}_{e,i}$, ${\rm W}_{f, i^\ast}$ and $\rho^\ast_e \alpha_i $ coincide with the functions  
$\chi_i(u_1)$, $\chi_i^-(u_2)$ and $\alpha_i(h)$ respectively under the isomorphism $\iota$ in \eqref{sfivqn}. In other words, the composition $\iota \circ \kappa$ keeps the functions $\chi_i(u_1)$, $\chi_i^-(u_2)$ and $\alpha_i(h)$ invariant. 

As shown above, the braid group action on $\mathscr{P}_{\G, \odot}$ induces a braid group action on $\G^\ast$. 
Following \cite[\S 7.5]{CKP}, the induced braid group action on $\G^\ast$ is the semiclassical limit of Lusztig's braid group action on $\mathcal{U}_q(\mathfrak{g})$. As a consequence, the maps $\iota$ and $\kappa$ are equivarient under the braid group actions on both algebras. Let $\sigma$ be a braid group element. If $f\in \mathcal{O}(\G^\ast)$ is invariant under the composition $\iota\circ \kappa$, then we have
\[
\iota \circ \kappa (\sigma^*f) = \sigma^* ( \iota \circ \kappa (f)) = \sigma^* f.
\]
Therefore, the functions $\chi_i(u_1)$, $\chi_i^-(u_2)$ and $\alpha_i(h)$ and their images under Braid group actions are invariant under $\iota\circ \kappa$. Note that all these functions generate $\mathcal{O}(\G^*)$ as an algebra. Therefore $\iota\circ \kappa$ is the identical map. 
\end{proof}

\subsection{Proof of Theorem \ref{main.result.2}}

Gross, Hacking, Keel, and Kontsevich constructed a formal basis (called $\Theta$-basis) for cluster Poisson algebras. If the underlying quiver admits a {\it maximal green sequence} (or more generally a {\it reddening sequence}),  then the $\Theta$-basis is an actual linear basis for the cluster Poisson algebra (\cite[Prop.0.7]{GHKK}).

Now let us focus on the quantum cluster algebra $\mathcal{O}(\mathscr{P}_{\G, \odot})$. When $\G={\rm PGL}_n$, the existence of reddening sequences of $\mathcal{O}(\mathscr{P}_{\G, \odot})$ has been proved in \cite{GS2}.
For general $\G$, let us first describe an underlying quiver for $\mathcal{O}(\mathscr{P}_{\G, \odot})$. We take an ideal triangulation of $\odot$ and use the puncture as the chosen vertex for both triangles as follows
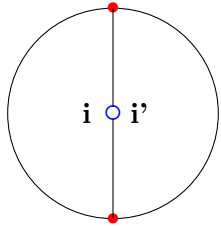
\begin{center}
\begin{tikzpicture}[scale=0.7]
\draw  (0,0) circle (20mm);
\node [blue, thick]  at (0,0) {\Large $\circ$};
\node [red]  at (0,2) {\small $\bullet$};
\node [red]  at (0,-2) {\small $\bullet$};
\node at (-.5,0) {\bf i };
\node at (.5,0) {\bf i'};
\draw (0,2)--(0,.12);
\draw (0,-2) -- (0, -.1);
    \end{tikzpicture}
 \end{center}
 We further choose reduced decompositions ${\bf i}$ and ${\bf i}'$ of $w_0$ for these two triangles. Fix a bipartite coloring on the Dynkin diagram for $\G$. Let $b\in W$ be the product black simple reflections and $w\in W$ the product of white ones. Consider the Coxeter element $c=bw$. If $\G$ is not of type $A_{2n}$, then the Coxeter number $h$ is even, and we use the reduced decomposition $w_0=c^{h/2}$ for both ${\bf i}$ and ${\bf i}'$.
 If $\G$ is of type $A_{2n}$, then the Coxeter number $h$ is odd. In this case, we will use the following two reduced decompositions as ${\bf i}$ and ${\bf i}'$:
 \[
 w_0= \underbrace{bwb\dots }_\text{$h$ factors}=  \underbrace{wbw\dots }_\text{$h$ factors}. 
 \]
 Putting them together, we get $w_0^2= c^h$ as braid group elements. Denote by $Q_{\bf i, {\bf i}'}$ the quiver of $\mathcal{P}_{\G, \odot}$ associated to the above data. By the amalgamation procedure, the mutable part of $Q_{\bf i, {\bf i}'}$ coincides with the quiver $Q_h(\mathfrak{g})$ in \cite{IIHI}. By Theorem 2 of \cite{IIHI}, the quiver $Q_{\bf i, {\bf i}'}$ admits a maximal green sequence. 

 Combining the aforementioned results, we obtain a linear basis $\Theta$ of $\mathcal{O}(\mathscr{P}_{\G, \odot})$.
The basis $\Theta$ admits positive integer structure coefficients. The quasi cluster transformations of $\mathscr{P}_{\G, \odot}$, which in particular contains the Weyl group and the Braid group actions, preserve the basis $\Theta$ as a set.

The Weyl group $W$ acts on the basis $\Theta$. For every $W$-orbit $c\subset \Theta$, we define 
\begin{equation}
\label{avfvfew}
\theta_c:= \sum_{\theta_i \in c} \theta_i.
\end{equation}
The functions \eqref{avfvfew} form a basis for $\mathcal{O}(\mathscr{P}_{\G, \bS})^W$.

Recall that the outer monodromies $\mu_i$ are Poisson central elements of $\mathcal{O}(\mathscr{P}_{\G, \odot})$. For $a=(a_1, \ldots, a_r)\in \mathbb{Z}^n$, let us set $\mu^a=\prod_{i=1}^r \mu_i^{a_i}$. 
\begin{lemma} \label{sdnnOASC}
If $\theta\in \Theta$, then $\theta \cdot \mu ^a \in \Theta$.
\end{lemma}
\begin{proof} By definition $\mu^a$ is a regular function of $\mathscr{P}_{\G, \odot}$. Meanwhile, $\mu^a$ are presented as Laurent monomials in every cluster chart of $\mathscr{P}_{\G, \odot}$. The functions $\mu_a$ are a special family of {\it global monomials} in terms of \cite[Def.0.1]{GHKK}. Therefore $\mu^a\in \Theta$. Let $\theta\in \Theta$. Then 
\be
\label{vcnfoi}
\theta \cdot \mu^a =\sum_{\theta_i} c(\theta,\mu^a;\theta_i) \theta_i,
\ee
where the structure coefficients $c(\theta,\mu^a;\theta_i)$ are non-negative integers. Meanwhile $\mu^{-a}\in \Theta$. Therefore
\be
\label{vcnfoi2}
\theta =\sum_{\theta_i} c(\theta,\mu^a;\theta_i) \theta_i \cdot \mu^{-a} =\sum_{\theta_i}\sum_{\theta_j}c(\theta,\mu^a;\theta_i)c(\theta_i, \mu^{-a}, \theta_j) \theta_j.
\ee
By comparing both sides of \eqref{vcnfoi2}, we see that there is only term on the right hand side of \eqref{vcnfoi}. Therefore $\theta \cdot \mu^a\in \Theta$.
\end{proof}
By Lemma \ref{sdnnOASC}, we obtain a $\mathbb{Z}^n$ action on $\Theta$. The functions $\mu_i$ are invariant under the $W$-action. Therefore the induced $\mathbb{Z}^n$ action commutes with the $W$-action, and descends to an $\mathbb{Z}^n$-action on the set of functions \eqref{avfvfew}. After quotient out the ideal $\mathcal{I}$, 
the functions $\theta_c$ in the same $\mathbb{Z}^n$ orbit descends to one function on $\mathcal{O}(\G^\ast)$. Putting them together, we obtain a basis $\overline{\Theta}$ for $\mathcal{O}(\G^\ast)$.

The basis $\Theta$ of $\mathcal{O}(\mathscr{P}_{\G, \odot})$ has positive integer structure coefficients and satisfies an invariance property with respect to the braid group action on $\mathcal{O}(\mathscr{P}_{\G, \odot})$. Following its construction, the basis $\overline{\Theta}$ of $\mathcal{O}(\G^\ast)$ inherits both properties from $\Theta$, which completes the proof of Theorem \ref{main.result.2}.

\vskip 2mm

We conjecture that Thereom \ref{main.result.2} can be generalized to the quantum setting. 
\begin{example} \label{acsdnjoqde}
We continue from Example \ref{scvdnfqpvnwqn}. Recall the Casimir element
\[
{\bf C}:={\bf E}{\bf F}-q{\bf K}^{-1} - q^{-1} {\bf K}= {\bf F}{\bf E}- q^{-1}{\bf K}^{-1} -q{\bf K}.
\]
The element ${\bf C}$ generates the center of $\mathcal{U}_q(\mathfrak{sl}_2)$.
Let $T_0, T_1, \ldots$ be the sequence of Chebyshev polynomials given by $T_0=1$ and $T_n(t+t^{-1})=t^n+t^{-n}$ for $n >0$. The set
\[
\Theta=\left\{q^{lm}{\bf E}^l{\bf K}^mT_n({\bf C})~\middle|~ l\geq 0, n\geq 0, m \in \mathbb{Z}\right\} \bigsqcup \left\{q^{ml}{\bf K}^l{\bf F}^m T_n({\bf C})~\middle|~ l\in \mathbb{Z}, m>0, n\geq 0 \right\}
\]
forms a linear basis of $\mathcal{U}_q(\mathfrak{sl}_2)$ with structure coefficients in $\mathbb{N}[q, q^{-1}]$.
\end{example}

\begin{remark} The above basis $\Theta$ for $\mathcal{U}_q(\mathfrak{sl}_2)$ is constructed using the lamination model of Fock and Goncharov in \cite[\S 12]{FGteich}. In this case, the basis $\Theta$ coincides with the double canonical bases of Berenstein and Greenstein \cite{BG}. The author is grateful to Yiqiang Li for bringing his attention to the paper \cite{BG}.
For higher rank groups, the actual calculation of theta bases in \cite{GHKK} requires counting an enormous amount of broken lines in large dimensional Scattering diagrams. To give an explicit construction of the basis for  $\mathcal{O}_q(\mathscr{P}_{\G, \bS})$ via concrete combinatorial models is an interesting direction for future research. See \cite{DS} and \cite{K} for recent results regarding the ${\rm SL}_3$ case.
\end{remark}

\end{document}